\documentclass[a4paper,12pt]{article}
\usepackage{amsmath, amsfonts, amssymb, amsthm}
\usepackage[english]{babel}

\newcounter{num}[section]

\newenvironment{theorem}
{\refstepcounter{num}%
\bigskip\noindent\nopagebreak[4]{\bf Theorem~\arabic{section}.\arabic{num}. }\it}
\newenvironment{corollary}
{\refstepcounter{num}%
\bigskip\noindent\nopagebreak[4]{\bf Corollary~\arabic{section}.\arabic{num}. }\it}
\newenvironment{lemma}
{\refstepcounter{num}%
\bigskip\noindent\nopagebreak[4]{\bf Lemma~\arabic{section}.\arabic{num}. }\it}
\newcommand{\LL}{{\mathcal{L}}}

\newcommand{\Ss}{{\mathbf{S}}}
\newcommand{\V}{{\mathrm{V}}}

\newcommand{\pr}{{\prime}}

\newcommand{\al}{{\alpha}}

\newcommand{\M}{{\mathcal{M}}}

\renewcommand{\c}{{\mathbf{c}}}

\newcommand{\1}{{^{-1}}}

\begin{document}

\author{Artem N. Shevlyakov}
\title{On disjunction of equations in completely regular semigroups}

\maketitle

\abstract{A semigroup $S$ is called an equational domain if any finite union of algebraic sets over $S$ is algebraic. We prove if a completely regular semigroup $S$ is an equational domain then $S$ is completely simple.}

\section{Introduction}

In~\cite{AG_over_groupsI,AG_over_groupsII} the main definitions of algebraic geometry over groups were given. Following these papers, an equation over a group $G$ is an equality $w(X)=1$, where $w(X)$ is an element of the free product $G\ast F(X)$. An algebraic set over $G$ is a solution set of a system of equations over $G$. 

The union of finite number of algebraic sets is not necessary algebraic over $G$. However, there exist groups (which were completely described in~\cite{uniTh_IV}), where any finite union of algebraic sets is algebraic. Following~\cite{uniTh_IV}, the groups with such property are called equational domains.  

The class of semigroups is more wider than the class of groups. Thus, it is naturally to search equational domains in the varieties of semigroups which are close to groups. One of the ``group-like'' class is the variety of completely regular semigroups, since any completely regular semigroup is the union of its maximal subgroups. 

In the current paper we prove that completely regular equational domain is completely simple.

\section{Definitions of semigroup theory}

Let us give the necessary definitions of semigroup theory. For a detailed introduction in semigroup theory one can recommend~\cite{howie}.

A semigroup $S$ with a unique two-sided ideal $I=S$ is called \textit{simple}.
A simple semigroup $S$ is \textit{completely simple} if it has minimal left and right ideals.

A semigroup $S$ is {\it completely regular (c.r.)} if it is a union of its maximal subgroups.  It is known that any completely simple semigroup is c.r.

A \textit{semilattice} $\Omega$ is a commutative idempotent semigroup, i.e. $\Omega$ satisfies the identities
\[
xy=yx,\; xx=x.
\] 

One can define a partial order $\leq$ over a semilattice $\Omega$ by
\[
x\leq y \Leftrightarrow xy=x.
\]

Let $\Omega$ be a semilattice. A disjount union of semigroups $S=\{S_\al|\al\in\Omega\}$ is called a \textit{semilattice of semigroups} if for any pair $\al,\beta\in\Omega$, $\al\leq\beta$ there exists a homomorphism $\psi_{\al,\beta}\colon S_\beta\to S_\al$ such that 
\begin{enumerate}
\item $\psi_{\al,\al}$ is trivial;
\item $\psi_{\al,\beta}\circ\psi_{\beta,\gamma}=\psi_{\al,\gamma}$ for all $\al\leq\beta\leq\gamma$;
\item $\psi_{\gamma\beta,\gamma}\circ\psi_{\gamma,\al}=\psi_{\gamma\beta,\beta}\circ\psi_{\beta,\al}$ for all $\beta,\gamma\leq\al$;
\item the product of elements $s_1\in S_\al$, $s_2\in S_\beta$ is defined by
\begin{equation}
\label{eq:multiplication_in_lattice_of_semigroups}
s_1s_2=\psi_{\al\beta,\al}(s_1)\psi_{\al\beta,\beta}(s_2).
\end{equation}
\end{enumerate}

The next theorem establishes the connections between completely regular semigroups and semilattices.

\begin{theorem}\textup{\cite{clifford_clifford_semigroups}}
\label{th:clifford_semigroups_structure}
Any c.r. semigroup is isomorphic to a semilattice of completely simple semigroups.
\end{theorem}

\section{Algebraic geometry over semigroups}

All definitions below are derived from the general notions of~\cite{uniTh_I}, where the definitions of algebraic geometry were formulated for an arbitrary algebraic structure in the language with no predicates.

As the inversion ${ }\1$ (the taking the inverse element in the corresponding maximal subgroup) is algebraic in any c.r. semigroup, one can consider the language $\LL_0=\{\cdot,\1\}$. For a given c.r. semigroup $S$ one can extend $\LL_0$ by new constants $\{s|s\in S\}$ which correspond to the elements of the semigroup $S$. The obtained language is denoted by $\LL_S$ and further all semigroups are considered in such language.

Let $X$ be a finite set of variables $x_1,x_2,\ldots,x_n$. \textit{A term} of a language $\LL_S$ ($\LL_S$-term) in variables $X$ is one of the next expressions:
\begin{enumerate}
\item variable $x_i$;
\item constant $s$;
\item the product of two terms;
\item $(t(X))\1$, where $t(X)$ is a term.
\end{enumerate}

For example, the expressions $xs(y^2x)\1$, $((xs_1\1 y)\1 s_2)\1 x^2$ are $\LL_S$-terms.

{\it An equation} over $\LL_S$ is an equality of two $\LL_S$-terms $t(X)=s(X)$. {\it A system of equations} over $\LL_S$ ({\it a system} for shortness) is an arbitrary set of equations over $\LL_S$.
 
A point $P=(p_1,p_2,\ldots,p_n)\in S^n$ is a \textit{solution} of a system $\Ss$ in variables $x_1,x_2,\ldots,x_n$, if the substitution $x_i=p_i$ reduces any equation of $\Ss$ to a true equality in the semigroup $S$. The set of all solutions of a system $\Ss$ in the semigroup $S$ is denoted by $\V_S(\Ss)$. A set $Y\subseteq S^n$ is called  {\it algebraic} over the language $\LL_S$ if there exists a system over $\LL_S$ in variables $x_1,x_2,\ldots,x_n$ with the solution set $Y$. 

Following~\cite{uniTh_IV}, let us give the main definition of our paper.

A c.r. semigroup $S$ is an {\it equational domain} ({\it e.d.} for shortness) in the language $\LL_S$ if for any finite set of algebraic sets $Y_1,Y_2,\ldots,Y_n$ over $\LL_S$ the union $Y=Y_1\cup Y_2\cup\ldots\cup Y_n$ is algebraic.

\section{Main result}
\label{sec:semilattices_of_semigroups}
%чем меньше индекс, тем больше полугруппа

According Theorem~\ref{th:clifford_semigroups_structure}, we denote a c.r. semigroup by $S=\{S_\al|\al\in\Omega\}$, where $\Omega$ is a semilattice and $S_\al$ are completely simple semigroups.

\begin{lemma}
\label{l:value_of_term_in_semilattice}
Let $S=\{S_\al|\al\in\Omega\}$ be a c.r. semigroup. Suppose an $\LL_S$-term $t(x)$ contains the constants from the subsemigroups with indexes $\al_1,\al_2,\ldots,\al_n$. Hence, the value of $t(x)$ at the point $x\in S_\beta$ belongs to the subsemigroup $S_\gamma$, where 
\[
\gamma=\al_1\al_2\ldots\al_n\beta.
\]
\end{lemma}
\begin{proof}
Let us prove the statement of the lemma by the induction on the definition of the term $t(x)$. If  $t(x)$ is either variable or constant the lemma obviously holds.

Let $t(x)=(t^\pr(x))\1$. By the assumption of the induction, the value of $t^\pr(x)$ belongs to $S_{\gamma^\pr}$, where 
\[
\gamma^\pr=\al_1^\pr\al_2^\pr\ldots\al_n^\pr\beta 
\]
and $t^\pr(x)$ contains the constants from the subsemigroups $S_{\al_i^\pr}$ ($1\leq i\leq n$).
As the subsemigroup $S_{\gamma^\pr}$ is c.r., $t(x)\in S_{\gamma^\pr}$ and the lemma is proved.

Suppose now $t(x)=t_1(x)t_2(x)$ and the value of the term $t_i(x)$ belong to the subsemigroup  $S_{\gamma_i}$, where
\[
\gamma_i=\al_{i1}\al_{i2}\ldots\al_{in_i}\beta,
\]
and the term $t_i(x)$ contains the constants from the subsemigroups $S_{\al_{ij}}$.

By the definition of the multiplication~(\ref{eq:multiplication_in_lattice_of_semigroups}) the value of the term $t(x)$ belongs to $S_{\gamma_1\gamma_2}$, where the index
\[
\gamma_1\gamma_2=\al_{11}\al_{12}\ldots\al_{1n_1}\al_{21}\al_{22}\ldots\al_{2n_2}\beta,
\]
contains $\beta$ and the indexes of all constants occurring in $t(x)$. 
\end{proof}

\begin{lemma}
\label{l:for_bands_hom}
Let $S=\{S_\al|\al\in\Omega\}$ be a c.r. semigroup, $\al,\beta\in\Omega$, $\al\leq\beta$, $b\in S_\beta$ and $a=\psi_{\al,\beta}(b)\in S_\al$. Hence for an $\LL_S$-term $t(x)$ with $t(b)\in S_\delta$, $t(a)\in S_\gamma$ it holds $\gamma\leq \delta$ and, moreover,
\begin{equation}
\psi_{\gamma,\delta}(t(b))=t(a).
\label{eq:hom_between_values}
\end{equation}
\end{lemma}
\begin{proof}
The inequality $\gamma\leq\delta$ immediately follows from Lemma~\ref{l:value_of_term_in_semilattice}. 

Let us prove~(\ref{eq:hom_between_values}) by the induction on the definition of the term $t(x)$.

If $t(x)$ is either a constant or a variable the equality~(\ref{eq:hom_between_values}) is obviously holds. 

Let $t(x)=(t^\pr(x))^{-1}$, and $t^\pr(x)$ satisfies $t^\pr(b)\in S_\delta$, $t^\pr(a)\in S_\gamma$, $\gamma\leq\delta$. By the definition of the inversion, we have $t(b)\in S_\delta$, $t(a)\in S_\gamma$.  Thus
\[
\psi_{\gamma,\delta}(t(b))=\psi_{\gamma,\delta}((t^\pr(b))^{-1})=(\psi_{\gamma,\delta}(t^\pr(b)))^{-1}=
(t^\pr(a))^{-1}=t(a).
\]

Suppose now $t(x)=t_1(x)t_2(x)$, where
\[
t_1(a)\in S_{\gamma_1},\; t_1(b)\in S_{\delta_1},\; \gamma_1\leq\delta_1,\; t_1(a)=\psi_{\gamma_1,\delta_1}(t_1(b)),
\]
\[
t_2(a)\in S_{\gamma_2},\; t_2(b)\in S_{\delta_2},\; \gamma_2\leq\delta_2,\; t_2(a)=\psi_{\gamma_2,\delta_2}(t_2(b)).
\]
According Lemma~\ref{l:value_of_term_in_semilattice}, we have $\gamma=\gamma_1\gamma_2$, $\delta=\delta_1\delta_2$, and, hence, $\gamma\leq\delta$.

Finally, we obtain
\begin{multline*}
\psi_{\gamma,\delta}(t(b))=\psi_{\gamma,\delta}(t_1(b)t_2(b))=
\psi_{\gamma,\delta}(\psi_{\delta,\delta_1}(t_1(b))\psi_{\delta,\delta_2}(t_2(b)))=\\
\psi_{\gamma,\delta}(\psi_{\delta,\delta_1}(t_1(b)))\psi_{\gamma,\delta}(\psi_{\delta,\delta_2}(t_2(b)))=
\psi_{\gamma,\delta_1}(t_1(b))\psi_{\gamma,\delta_2}(t_2(b))=\\
\psi_{\gamma_1\gamma_2,\delta_1}(t_1(b))\psi_{\gamma_1\gamma_2,\delta_2}(t_2(b))=
\psi_{\gamma_1\gamma_2,\gamma_1}(\psi_{\gamma_1,\delta_1}(t_1(b)))\psi_{\gamma_1\gamma_2,\gamma_2}(\psi(t_2(b)))=\\
\psi_{\gamma_1\gamma_2,\gamma_1}(t_1(a))\psi_{\gamma_1\gamma_2,\gamma_2}(t_2(a))=t_1(a)t_2(a)=t(a),
\end{multline*}
that proves the lemma.
\end{proof}

\begin{lemma}
\label{l:for_bands_pre}
Let $S=\{S_\al|\al\in\Omega\}$ be a c.r. semigroup, $\al,\beta\in\Omega$, $\al\leq\beta$, and an $\LL_S$-term $t(x)$ contains a constant from the subsemigroup $S_\al$. Therefore, for any $b\in S_\beta$ the elements $t(b),t(\psi_{\al,\beta}(b))$ belongs to $S_\gamma$ for some $\gamma\in\Omega$.
\end{lemma}
\begin{proof}
Denote $a=\psi_{\al,\beta}(b)$. The proof of this lemma follows from Lemma~\ref{l:value_of_term_in_semilattice} which states that the elements $t(b),t(a)$ belongs to the subsemigroups with indexes
\[
\beta\al_1\al_2\ldots\al_n,\; \al\al_1\al_2\ldots\al_n,
\]
where the constants of the term $t(x)$ belong to the subsemigroups with the indexes $\al_i$.

by the condition of the lemma, the term $t(x)$ contains a constant from the subsemigroup $S_\al$. Without loss of generality, one can put $\al_1=\al$. Hence, the element $t(b)$ belongs to the subsemigroup with the index 
\[
\beta\al\al_2\ldots\al_n=\al\al_2\ldots\al_n.
\]

Obviously, the element $t(a)$ belongs to the subsemigroup with the same index:
\[
\al\al\al_2\ldots\al_n=\al\al_2\ldots\al_n.
\]
\end{proof}

\begin{lemma}
\label{l:for_bands}
Let $S=\{S_\al|\al\in\Omega\}$ be a c.r. semigroup, $\al,\beta\in\Omega$, $\al\leq\beta$. Suppose an $\LL_S$-term $t(x)$ contains a constant from the subsemigroup $S_\al$. Hence, for any $b\in S_\beta$ it holds
\[
t(b)=t(\psi_{\al,\beta}(b)).
\]
\end{lemma}
\begin{proof}
Denote $a=\psi_{\al,\beta}(b)$. By Lemma~\ref{l:for_bands_pre}, the elements $t(b),t(a)$ belong to the same subsemigroup $S_\gamma$. Following Lemma~\ref{l:for_bands_hom}, there exists a homomorphism $\psi_{\gamma,\gamma}(t(b))=t(a)$. By the definition of a semilattice of semigroups, the homomorphism $\psi_{\gamma,\gamma}$ is trivial. Hence, $t(a)=t(b)$.
\end{proof}

\begin{theorem}
\label{th:about_semilattices}
Any c.r. semigroup $S=\{S_\al|\al\in\Omega\}$ with $|\Omega|>1$ is not an equational domain in the language $\LL_S$.
\end{theorem}
\begin{proof}
Let $\al,\beta\in\Omega$, $\al<\beta$, $b\in S_\beta$. By $a\in S_\al$ denote $\psi_{\al\beta}(b)$. 

Assume the converse: the set 
\[
\M=\{(x,y)|x=b\mbox{ or } y=b\}=\V_S(x=b)\vee\V_S(y=b)
\]
is the solution set of a system $\Ss$. As $P=(a,a)\notin\M$, there exists an equation $t(x,y)=s(x,y)\in \Ss$ with $t(a,a)\neq s(a,a)$.

Observe that the equation $t(x,y)=s(x,y)$ should contain the occurrences of the both variables $x,y$. Otherwise, if it depends on a single variable $x$ (or $y$) we have $t(a)=s(a)$, as  $(a,b)\in\M$ (or $(b,a)\in\M$). However, by the choice of the equation, $t(a)\neq s(a)$, and we came to the contradiction.

Thus, the equation $t(x,y)=s(x,y)$ has one of the following types:
\begin{enumerate}
\item any part of the equation $t(x,y)=s(x,y)$ contains the both variables $x,y$;
\item the first part of the equation depends on a single variable, whereas the second one depends on two variables; 
\item the first part of the equation does not contain $y$, and the second part does not contain $x$;
\item the first part of the equation depends on the both variables, but the second part does not contain any variable. 
\end{enumerate}

Let us consequently consider all types of the equation $t(x,y)=s(x,y)$.

\begin{enumerate}
\item Let $t^\pr(x)=t(x,a)$, $s^\pr(x)=s(x,a)$. As $(b,a)\in\M$, then $t^\pr(b)=s^\pr(b)$. 
By Lemma~\ref{l:for_bands}, we have $t^\pr(a)=t^\pr(b)$, $s^\pr(a)=s^\pr(b)$, hence $t^\pr(a)=s^\pr(a)$, that contradicts with the choice of $t(x,y)=s(x,y)$.
\item In this case we have the equation $t(x)=s(x,y)$. According Lemma~\ref{l:for_bands}, we have $s(a,b)=s(a,a)$. As $(a,b)\in\M$, we obtain $t(a)=s(a,b)$. Thus, $t(a)=s(a,a)$ and we came to the contradiction with the choice of the equation $t(x)=s(x,y)$.
\item For the equation $t(x)=s(y)$ we have $t(b)=s(a)$ (as $(b,a)\in\M$), $t(a)=s(b)$ (since $(a,b)\in\M$), $t(b)=s(b)$ (as $(b,b)\in\M$). Thus, $t(a)=s(a)$, that contradicts with the choice of the equation $t(x)=s(x)$
\item For $t(x,y)=\c$ we have $t(a,b)=\c$  ($(a,b)\in\M$). By Lemma~\ref{l:for_bands}, obtain $t(a,b)=t(a,a)=\c$, and come to the contradiction.
\end{enumerate}
\end{proof}

\begin{corollary}
\label{cor:clifford->completely_simple}
Is a c.r. semigroup $S$ is an equational domain in the language $\LL_S=\{\cdot,{ }^{-1}\}\cup\{s|s\in S\}$, then $S$ is completely simple.
\end{corollary}

The information of the author:

Artem N. Shevlyakov

Omsk Branch of Institute of Mathematics, Siberian Branch of the Russian Academy of Sciences

644099 Russia, Omsk, Pevtsova st. 13

Phone: +7-3812-23-25-51.

e-mail: \texttt{a\_shevl@mail.ru}
\end{document}